\documentclass[12pt, a4paper, twoside]{amsart}

\usepackage[margin=2.85cm]{geometry}

\usepackage[T1]{fontenc}
\usepackage{amsmath}
\usepackage{amsthm}
\usepackage{amsfonts}
\usepackage{bbm}
\usepackage[UKenglish]{babel}
\usepackage{enumerate}
\usepackage[dvips]{graphicx}
\usepackage{bm}
\usepackage{ae}
\usepackage{hyperref}
\usepackage{tikz}

\theoremstyle{definition}

\theoremstyle{remark}

\theoremstyle{plain}
\newtheorem{thm}{Theorem}[section]
\newtheorem{lem}[thm]{Lemma}

\newtheorem{cor}[thm]{Corollary}

\newtheorem*{kgt}{Theorem(Khintchine-Groshev)}

\newcommand{\norm}[1]{\ensuremath{\left\Vert #1 \right\Vert}}
\newcommand{\abs}[1]{\ensuremath{\left\vert #1 \right\vert}}
\DeclareMathOperator{\diam}{diam}
\DeclareMathOperator{\mat}{Mat}

\DeclareMathOperator{\GL}{GL}

\newcommand{\mb}[1]{\mathbb{#1}}
\newcommand{\mbf}[1]{\mathbf{#1}}
\newcommand{\mcal}[1]{\mathcal{#1}}

\begin{document}

\title{Metrical results on systems of small linear forms}

\author{M. HUSSAIN}

\address{M. Hussain, Department of Mathematics and Statistics, La Trobe University, Melbourne, 3086, Victoria, Australia}

\email{M.Hussain@latrobe.edu.au}

\thanks{MH's visit to Aarhus was sponsored in part by La Trobe
  University's travel grant and from SK's Danish Research Council for
  Independent Research}

\author{S. KRISTENSEN}

\address{S. Kristensen, Department of Mathematical Sciences, Faculty
  of Science, University of Aarhus, Ny Munkegade 118,
  DK-8000 Aarhus C, Denmark}

\email{sik@imf.au.dk}

\thanks{SK's Research supported the Danish Research Council for
  Independent Research. }

\keywords{Diophantine approximation; systems of linear forms; absolute
  value}

\subjclass[2000]{11J83}

\begin{abstract}
  In this paper the metric theory of Diophantine approximation
  associated with the small linear forms is investigated.
  Khintchine--Groshev theorems are established along with Hausdorff
  measure generalization without the monotonic assumption on the
  approximating function.
\end{abstract}

\maketitle

\section{Introduction}

Let $\psi:\mb{N}\to\mb{R}^+ $ be a function tending to $0$ at
infinity referred to as an \emph{approximation} function. An $m\times
n$ matrix $X=(x_{ij})\in\mb{I}^{mn}:= [ 0,1] ^{mn}$ is said to be
\emph{$\psi$--approximable} if the system of inequalities
\[
|q_{_{1}}x_{_{1}i}+q_{_{2}}x_{_{2}i}+\dots+q_{m}x_{mi}| \leq
\psi(|\mbf{q}|)\text{\quad{}for\quad}(1\leq i\leq n),
\]
is satisfied for infinitely many $\mbf{q}\in \mb{Z}
^{m}\setminus\{\mbf 0\}$. Here and throughout, the system $ q_1x_{1i}+
\dots+ q_{m}x_{mi}$ of $n$ linear forms in $m$ variables will be
written more concisely as $\mbf{q}X$, where the matrix $X$ is regarded
as a point in $ \mb{I}^{mn}$ and $|\mbf{q}|$ denotes the supremum norm
of the integer vector $\mbf{q}$.  The set of $\psi$--approximable
points in $\mb{I}^{mn} $ will be denoted by $W_0( m,n;\psi)$;
\[
W_0( m,n;\psi) :=\{ X\in
\mb{I}^{mn}:|\mbf{q}X|<\psi(|\mbf{q}|)\text{\ for i.m.\ }\mbf{q}\in
\mb{Z}^{m}\setminus \{\mbf{0}\}\},
\]
where `i.m.' means `infinitely many'.  For a monotonic approximating
function, the metric theory has been established for the set $W_0(m,
n; \psi)$ in \cite{mhjl} (the dimension of this set was obtained in
\cite{Dickinson}) and it's generalization to mixed case in \cite{DH}.
The aim of this paper is to discuss the metric theory for the set
$W_0(m,n;\psi)$ without the monotonicity assumption on the
approximating function.

It is worth relating the above to the set of $\psi$--approximable
matrices as is often studied in classical Diophantine approximation.
In such a setting studying the metric structure of the $\limsup$-set
\[
W( m,n;\psi)=\{ X\in \mb{I}^{mn}:\| \mbf{q}X\|
<\psi(|\mbf{q}|)\text{ for i.m. \ }  \mbf{q}\in \mb{Z}^{m}\setminus
\{\mbf{0}\}\},
\]
where $\|x\|$ denotes the distance of $x$ to the nearest integer
vector, is a central problem and the theory is well established, see
for example \cite{BBDV, BDV, bovdod, Falc2, sprindzuk}. In the case
that the approximating function is monotonic, the main result in this
setting is the Khintchine-Groshev theorem which gives an elegant
answer to the question of the size of the set $W(m,n;\psi)$. The
result links the measure of the set to the convergence or otherwise of
a series that depends only on the approximating function and is the
template for many results in the field of metric number theory. The
following is an improved modern version of this fundamental result --
see \cite{BDV} and references within. Given a set $X$, $| \, X \,
|_{k}$ denotes $k$-dimensional Lebesgue measure of $X$.

\begin{kgt}\label{remo}

\noindent Let $\psi$ be an approximating function. Then

 \[| W\left( m,n;\psi \right) |_{mn} = \left\{
\begin{array}{cl}
0 &
{\rm \ if} \  \sum \limits_{r=1}^{\infty}r^{m-1}\psi (r)^{n}<\infty,\\
&\\
1& {\rm \ if} \ \sum_{r=1}^{\infty} \
 r^{m-1}\psi (r)^{n}=\infty\,  \text{ and $\psi$ is monotonic}.
\end{array}\right.\]

\end{kgt}

The convergence part is reasonably straightforward to establish from
the Borel--Cantelli Lemma and is free from any assumption on $\psi$.
The divergence part constitutes the main substance of the
Khintchine--Groshev theorem and involves the monotonicity assumption
on the approximating function. It is worth mentioning that in the
original statement of the theorem \cite{Groshev, K1, K} the stronger
hypothesis that $q^m \psi(q)^n$ is monotonic was assumed.  In the
one-dimensional case $(m = n = 1)$, it is well known that the
monotonicity hypothesis in the Khintchine-Groshev theorem is
absolutely crucial. Indeed, Duffin and Schaeffer
\cite{DuffinSchaeffer} constructed a non-monotonic function for which
$\sum_{q=1}^{\infty}\psi(q)$ diverges but $|W(1,1; \psi)|=0.$ In other
words the Khintchine-Groshev theorem is false without the monotonicity
hypothesis and the conjectures of Catlin \cite{cat} and Duffin and
Schaeffer \cite{DuffinSchaeffer} provide appropriate alternative
statements, see \cite{BBDV} for the details and generalizations of
Duffin-Schaeffer and Catlin conjectures to the linear forms. Beyond
the one-dimensional case the monotonicity assumption on the
approximating function is completely removed. The proof is attributed
to various authors for different values of $m$.  For $m = 1$,
Khintchine--Groshev theorem without the monotonicity of $\psi$ was
proved by Gallagher \cite{Gal}.  For $m = 2$, this was recently proved
by Beresenevich and Velani in \cite{bvclassical}. For $m \geq 3$ it
can be derived from Schmidt \cite[Theorem 2]{schmidt1} or
Sprindz\v{u}k's \cite[\S1.5, Theorem 15]{sprindzuk}.

It is readily verified that $W_0( 1,n;\psi) = \{\mbf 0\}$ as any
$x=(x_1,x_2,\dots,x_n)\in{}W_0(1,n;\psi)$ must satisfy the
inequality $\left\vert qx_{j}\right\vert < \psi(q) $  infinitely
often. As $\psi(q)\to 0$ as $q\to\infty$ this is only possible if
$x_j=0$ for all $j=1,2,\dots,n.$ Thus when $m=1$ the set $W_0(1, n;
\psi)$ is a singleton and must have both zero measure and dimension.
We will therefore assume that $m\geq 2.$

\medskip

\emph{Notation.} To simplify notation the Vinogradov symbols $\ll$ and $\gg$
will be used to indicate an inequality with an unspecified positive
multiplicative constant depending only on $m$ and $n$.
If $a\ll b$ and $a\gg b$  we write $a\asymp
b$, and say that the quantities $a$ and $b$ are comparable.
A \emph{dimension function} is an increasing continuous
 function $f:\mb{R}^+\rightarrow
\mb{R}^+$
such that $f(r)\to 0$ as $r\to 0$.
Throughout the paper, $\mcal{H}^f$ denotes the
$f$--dimensional Hausdorff measure
which will be fully defined in section \ref{hm}. Finally, for convenience,
for a  given
 approximating function $\psi,$ define the function
 $$\Psi(r):=\frac{\psi(r)}{r}.$$

\section{Statement of the Results}

The main results below depend critically on assumptions on $m$ and
$n$. In order to get beyond the Duffin--Schaeffer counterexample (see
below), we will always assume that $m+n > 3$. However, an additional
phenomenon occurs when the number of forms is greater than or equal to
to the number of variables ($m \leq n$), and we will have to treat
each case separately. Our first result concerns the case when the
number of variables exceeds the number of forms.

\begin{thm}\label{thmmgn}
  Let $m> n$, $m+n>3$ and $\psi$ be an approximating function. Let
  $f$, $r^{-n^2}f(r)$ and $ r^{-(m-n-1)n}f(r)$ be dimension functions
  such that $r^{-mn}f(r)$ is monotonic. Then
  \[\mcal{H}^{f}\left(W_0\left( m,n;\psi \right)
  \right) = \left\{
    \begin{array}{cl}
      0 &
      {\rm \ if} \  \sum
      \limits_{r=1}^{\infty}f(\Psi (r))\Psi(r)^{-(m-1)n}r^{m-1}<\infty,\\
      &\\
      \mcal{H}^f(\mb{I}^{mn})& {\rm \ if} \  \sum
      \limits_{r=1}^{\infty}f(\Psi (r))\Psi(r)^{-(m-1)n}r^{m-1}=\infty.
    \end{array}\right.\]
\end{thm}

\noindent As in most of the statements the convergence part is
reasonably straightforward to establish and is free from any
assumptions on $m, n$ and the approximating function. This fact was
already established in \cite[Theorem 4]{mhjl}.  It is the
divergence statement which constitutes the main substance and this is
where conditions come into play.

The requirement that $r^{-mn}f(r)$ be monotonic is a natural and not
particularly restrictive condition. Note that if the dimension
function $f$ is such that $r^{-mn}f(r)\to \infty$ as $r\to 0$ then
$\mcal{H}^f(\mb{I}^{mn})=\infty$ and Theorem \ref{thmmgn} is the
analogue of the classical result of Jarn\'{\i}k (see \cite{DV,
  Jarnik1}). In the case when $f(r) := r^{mn}$ the Hausdorff measure
$\mcal{H}^f$ is proportional to the standard $mn$--Lebesgue measure
supported on $\mb{I}^{mn}$ and the result is the natural analogue of
the Khintchine--Groshev theorem for $ W_0\left( m,n;\psi\right )$.
\medskip
\begin{cor}\label{cor1}
  \noindent Let $m > n$ and $m+n>3$. Let $\psi$ be an approximating
  function, then
 \[| W_0\left( m,n;\psi \right) |_{mn} = \left\{
   \begin{array}{cl}
     0 &
     {\rm \ if} \
     \sum\limits_{r=1}^\infty{}\psi(r)^nr^{m-n-1}<\infty,\\
     &\\
     1& {\rm \ if} \
     \sum\limits_{r=1}^\infty{}\psi(r)^nr^{m-n-1}=\infty.
   \end{array}\right.\]
\end{cor}

\medskip

In the results above, the condition $m+n> 3$ is absolutely necessary.
For $m = 1$ the set $W_0(1, n; \psi)$ is singleton as already
remarked. For $m=2, n=1$, the Duffin--Schaeffer counter example can be
exploited to show that there exists a function $\psi$ such that
\[\sum_{r=1}^\infty{}\psi(r)=
\infty\ \ \ \text{but} \ \ |W_0(2, 1; \psi)|_2=0.\]
Indeed, the Duffin--Schaeffer counter example provides us with a
function $\psi$, such that the set
\begin{equation*}
  \mathcal{DS} = \left\{ y \in \mathbb{R} : \abs{q y - p} < \psi(q)
    \text{ for infinitely many $p,q \in \mathbb{Z}$}\right\}
\end{equation*}
is Lebesgue null, while the sum $\sum \psi(r) = \infty$. Using this
function as a $\psi$ in the definition of $W_0(2,1;\psi)$ and assuming
the measure of the latter set to be positive, using the ideas below in
the proof of Theorem \ref{thmmgn}, this will imply that $\mathcal{DS}$
has positive Lebesgue measure.

For $m\leq n$ the conditions on the dimension function in Theorem
\ref{thmmgn} change. This change is due to the fact that if $X\in
W_0(m,n;\psi)$ and $m\le n$ then a linear system of equations given by
$X$ is over-determined and the set of solutions lies in a subset of
strictly lower dimension than $mn$. Hence, the corresponding set of
$\psi$-approximable systems of forms will concentrate on a lower
dimensional surface.  This is proved in \cite{mhjl} where it is shown
that for $m\leq n$ the set $W_0\left( m,n;\psi \right)$ lie on a
$(m-1)(n+1)$--dimensional hypersurface $\Gamma$.  Therefore,
naturally, we expect the analogue of Theorem \ref{thmmgn} holds on
$\Gamma.$

\begin{thm}\label{thmmlen}
  Let $2< m\leq n$ and $\psi$ be an approximating function. Let $f$,
  $r^{-n^2}f(r)$, $ r^{-(m-n-1)n}f(r)$ and $r^{-(n-m+1)(m-1)}f(r)$ be
  dimension functions such that $r^{-(m-1)(n+1)}f(r)$ is monotonic.
  Then
  \[
  \mcal{H}^f(W_0(m,n;\psi))=
  0 \ \ {\rm {if}} \ \
  \sum_{r=1}^{\infty}f(\Psi
  (r))\Psi(r)^{-(m-1)n}r^{m-1}<\infty.
  \]
  On the other hand, if
  \[\sum_{r=1}^{\infty}f(\Psi(r))\Psi(r)^{-(m-1)n}r^{m-1}=\infty,\]
  then
  \[\mcal{H}^f(W_0(m,n;\psi))  = \left\{
    \begin{array}{cl}
      \infty &
      {\rm \ if} \  r^{-(m-1)(n+1)}f(r)\to\infty\  \text{as}\
      r\to{}0,\\
      &\\
      K & {\rm \ if} \  r^{-(m-1)(n+1)}f(r)\to C\
      \text{as}\  r\to0,
    \end{array}\right.\]
  for some fixed constant  $0\leq C<\infty$, where $0<K<\infty$.
\end{thm}
Note that for a dimension function $f$ which satisfies
$r^{-(m-1)(n+1)}f(r)\to{}C>0$ as $r\to{}0$ the measure $\mathcal{H}^f$
is comparable to standard $(m-1)(n+1)$-dimensional Lebesgue measure
and in the case when $f(r)=r^{(m-1)(n+1)}$, we obtain the following
analogue of the Khintchine-Groshev theorem.
\begin{cor}\label{kgmleqn}
  \noindent Suppose $2<m\leq n$ and assume that the conditions of
  Theorem~\ref{thmmlen} hold for the dimension function
  $f(r):=r^{(m-1)(n+1)}$. Then
  \[
  |W_0 (m,n;\psi)|_{(m-1)(n+1)}=\left\{\begin{array}{ll}
      0 & {\rm \ if}\ \ \sum_{r=1}^{\infty}\psi (r)^{m-1}<\infty, \\
      &\\
      K & {\rm \ if}\ \ \sum_{r=1}^{\infty}\psi (r)^{m-1}=\infty,
    \end{array}\right.
  \]
  where $0<K<\infty$.
\end{cor}

\section{Machinery}

The machinery required for the proofs of both the theorems is the
Mass Transference Principle along with `slicing' technique. We
merely state the results and refer the reader to \cite {slicing} for
further details.

\subsection{Hausdorff Measure and Dimension}
\label{hm}

Below is a brief introduction to Hausdorff $f$--measure and dimension.
For further details see \cite{berdod, Falc2}.  Let $F\subset
\mb{R}^n$.  For any $\rho>0$ a countable collection $\{B_i\}$ of balls
in $\mb{R}^n$ with diameters $\mathrm{diam} (B_i)\le \rho$ such that
$F\subset \bigcup_i B_i$ is called a $\rho$--cover of $F$.  Define
\[
\mcal{H}_\rho^f(F)=\inf \sum_if(\mathrm{diam}(B_i)),
\]
where the infimum is taken over all possible $\rho$--covers of $F$.
The Hausdorff $f$--measure of $F$ is
\[
\mcal{H}^f(F)=\lim_{\rho\to 0}\mcal{H}_\rho^f(F).
\]
In the particular case when $f(r)=r^s$ with $ s>0$, we write
$\mcal{H}^s$ for $\mcal{H}^f$ and the measure is referred to as
$s$--dimensional Hausdorff measure. The Hausdorff dimension of $F$ is
denoted by $\dim F $ and is defined as
\[
\dim F :=\inf\{s\in \mb{R}^+\;:\; \mcal{H}^s(F)=0\}.
\]

\subsection{Slicing}\label{sslicing}

We now state a result which is the key ingredient in the proof of
Theorems \ref{thmmgn} and \ref{thmmlen}. The result was used in
\cite{slicing} to prove the Hausdorff measure version of the W. M.
Schmidt's inhomogeneous linear forms theorem in metric number theory.
The authors refer to the technique as `slicing'.  Before we state the
result it is necessary to introduce a little notation.

Suppose that $V$ is a linear subspace of $\mb{R}^k$, $V^{\perp}$ will
be used to denote the linear subspace of $\mb{R}^k$ orthogonal to $V$.
Further $V+a:=\left\{v+a:v\in V\right\}$ for $a\in V^{\perp}$.

\medskip

\begin{lem}\label{slicing}
  Let $l, k \in \mb{N}$ be such that $l\leq k$ and let $f
  \;\textrm{and}\; g:r\to r^{-l}f(r)$ be dimension functions. Let
  $B\subset \mb{I}^k$ be a \emph{Borel} set and let $V$ be a
  $(k-l)$--dimensional linear subspace of $\mb{R}^k$. If for a subset
  $S$ of $V^{\perp}$ of positive $\mcal{H}^l$
  measure$$\mcal{H}^{g}\left(B\cap(V+b)\right)=\infty \hspace{.5cm}
  \forall \; b\in S,$$ \noindent then $\mcal{H}^{f}(B)=\infty.$
\end{lem}

\subsection{A Hausdorff measure version of Khintchine--Groshev
  theorem}

As an application of the mass transference principle for system of
linear forms developed in \cite {slicing} the Hausdorff measure
version of the Khintchine--Groshev theorem is established without the
monotonic assumption on the approximating function in \cite[Theorem
15]{BBDV}. The additional assumption that $\psi$ is monotonic was
assumed in \cite{BBDV} for the case $m=2$, but subsequently removed in
\cite{bvclassical}.

\begin{thm}\label{hmkg}
  Let $\psi$ be an approximating function and $m+n>2$. Let $f$ and
  $r^{-(m-1)n}f(r)$ be dimension functions such that $r^{-mn}f(r)$ is
  monotonic. Then, $$ \mcal{H}^f\left( W\left(m,n;\psi\right) \right)
  = \left\{
    \begin{array}{cl}
      0& {\rm \ if}  \qquad\displaystyle
      \sum \limits_{r=1}^{\infty}f(\Psi
      (r))\Psi(r)^{-(m-1)n}r^{m+n-1}<\infty \, \\
      \mcal{H}^f(\mb{I}^{mn})  & { \rm \ if}  \qquad \displaystyle
      \sum_{r=1}^{\infty} \
      f(\Psi
      (r))\Psi(r)^{-(m-1)n}r^{m+n-1}=\infty\, .
    \end{array}
  \right.
  $$
\end{thm}

\noindent Theorem \ref {hmkg} along with Lemma \ref{slicing} will be
used to prove the infinite measure case of Theorem \ref{thmmgn}.

\section{Proof of Theorem \ref{thmmgn}}
\label{sec:divergence}

As stated earlier the condition $r^{-mn}f(r)$ is not a restrictive
condition. The statement of the Theorem essentially reduces to two
cases, finite measure case, \emph{i.e.}, when $r^{-mn}f(r)\to C>0$ as
$r\to 0$ and to the infinite measure case which corresponds to
$r^{-mn}f(r)\to \infty$ as $r\to 0$. Therefore, we split the proof of
the Theorem \ref{thmmgn} into two parts, the finite measure case and
the infinite measure case.

Before proceeding, we will need the following key lemma, which will
make our proofs work.

\begin{lem}
  \label{lem:bi-lip}
  Let $S \subseteq \mat_{(m-n)\times n}(\mathbb{R})$ of full Lebesgue
  measure. Let $A \subseteq \GL_{n\times n}(\mathbb{R})$ be a set of
  positive Lebesgue measure. Then, the set
  \begin{equation*}
    \Lambda = \left\{
      \begin{pmatrix}
        X \\
        XY
      \end{pmatrix}
      \in \mat_{m \times n}(\mathbb{R}) : X \in A, Y \in S \right\}
  \end{equation*}
  has full Lebesgue measure inside $A \times S$.
\end{lem}

\begin{proof}
  Without loss of generality, we will assume that $\abs{A}_{n^2} <
  \infty$.  If this is not the case, we will take a subset of $A$.
  Suppose now for a contradiction that $\abs{(A \times S) \setminus
    \Lambda} > 0$ and let $Z$ be a point of metric density for this
  set. We will show that the existence of such a point violates the
  condition that $S$ is full.

  Fix an $\epsilon > 0$. There is a $\delta > 0$ such that
  \begin{equation*}
    \frac{\abs{\Lambda \cap B(Z,
        \delta)}}{\abs{B(Z, \delta)}} < \frac{\epsilon}{2^{mn+1}},
  \end{equation*}
  where $B(Z, \delta)$ denotes the ball centred at $Z$ of radius
  $\delta$. By definition of the Lebesgue measure, we may take a cover
  $\mathcal{C}$ of $\Lambda \cap B(Z, \delta)$ by hypercubes in
  $\mathbb{R}^{mn}$ such that
  \begin{equation*}
    \sum_{C \in \mathcal{C}} \diam(C)^{mn} < \frac{\epsilon}{2^{mn}}
    \abs{B(Z, \delta)} = \epsilon \delta^{mn}.
  \end{equation*}
  The latter equality follows as we are working in the supremum norm,
  so that a ball of radius $\delta$ is in fact a hypercube of side
  length $2\delta$.  We let $A_0 \subseteq A$ be the set of those $X
  \in A$ for which there is a $Y \in S$ such that $\binom{X}{XY} \in
  B(Z, \delta)$.  Note that by Fubini's Theorem $A_0$ has positive
  Lebesgue measure. In fact, the measure is equal to $2^{n^2}
  \delta^{n^2}$.

  For any $X \in A_0$ we define the set
  \begin{equation*}
    B(X) = \left\{
      \begin{pmatrix}
        X \\
        XY
      \end{pmatrix}
      \in B(Z, \delta) : Y \in S \right\}.
  \end{equation*}
  Note that
  \begin{equation*}
    \mathcal{C}(X) =  \left\{ \left(
      \begin{pmatrix}
        X \\
        \mat_{(m-n) \times n} (\mathbb{R})
      \end{pmatrix} \cap C\right)
      \in \mat_{m \times n}(\mathbb{R}) : C \in \mathcal{C} \right\}.
  \end{equation*}
  is a cover of  $B(X)$ by $(m-n)n$-dimensional hypercubes.

  As in \cite{mh}, we define for each $C \in \mathcal{C}$ a function,
  \begin{equation*}
    \lambda_C(X) =
    \begin{cases}
      1 & \text{if } \left(
        \begin{pmatrix}
          X \\
          \mat_{(m-n) \times n} (\mathbb{R})
        \end{pmatrix} \cap C\right) \neq \emptyset \\
      0 & \text{otherwise.}
    \end{cases}
  \end{equation*}
  It is easily seen that
  \begin{equation*}
    \int_{A_0} \lambda_C(X) dX \leq \diam(C)^{n^2},
  \end{equation*}
  where the integral is with respect to the $n\times n$-dimensional
  Lebesgue measure. Also,
  \begin{equation*}
    \sum_{C \in \mathcal{C}(X)} \diam(C)^{(m-n)n} =
    \sum_{C \in \mathcal{C}} \lambda_C(X) \diam(C)^{(m-n)n}.
  \end{equation*}

  We integrate the latter expression with respect to $X$ to obtain
  \begin{multline*}
    \int_{A_0} \sum_{C \in \mathcal{C}(X)} \diam(C)^{(m-n)n} dX =
    \sum_{C \in \mathcal{C}}\int_{A_0} \lambda_C(X) dX
    \diam(C)^{(m-n)n} \\
    \leq \sum_{C \in \mathcal{C}} \diam(C)^{mn} < \epsilon \abs{B(Z,
      \delta)}.
  \end{multline*}
  Since the right hand side is an integral of a non-negative function
  over a set of positive measure, there must be an $X_0 \in A_0$ with
  \begin{equation}
    \label{eq:3}
    \sum_{C \in \mathcal{C}(X_0)} \diam(C)^{(m-n)n} <
    \frac{\epsilon \abs{B(Z, \delta)}}{\mu(A_0)} = \frac{\epsilon
      \delta^{mn}}{2^{n^2} \delta^{n^2}} = \frac{\epsilon}{2^{n^2}}
    \delta^{n(m-n)}.
  \end{equation}
  Indeed, otherwise
  \begin{equation*}
    \int_A \sum_{C \in \mathcal{C}(X)} \diam(C)^{(m-n)n} dX
    \geq \int_A \frac{\epsilon\abs{B(Z, \delta)}}{\mu(A)} dX =
    \epsilon.
  \end{equation*}

  We may now estimate the $(m-n)n$-dimensional measure of $B(X_0)$
  from above by this sum.  This gives an upper estimate on the measure
  of $B(I_n)$, as $X_0$ is invertible. Furthermore, this estimate can
  be made arbitrarily small. But $B(I_n)$ is a cylinder set over $S$,
  so this is a clear contradiction since $S$ was assumed to be full.
\end{proof}

In applications, we will apply Lemma \ref{lem:bi-lip} with the set $S$
being $W(m-n,n,;\psi)$. This set is however a subset of
$\mathbb{I}^{(m-n)n}$, and so not full within $\mat_{(m-n)\times
  n}(\mathbb{R})$. It is however invariant under tranlation by integer
vectors, so this causes no loss of generality.

\subsection{Finite measure}

In order to proceed, we will make some restrictions. Let $\epsilon >
0$ and $N > 0$ be fixed but arbitrary. It is to be understood that
$\epsilon$ will be small eventually and $N$ large. We will define a
set $A_{\epsilon, N}$ of $m \times n$-matrices which is smaller than
the whole, but which tends to the whole set as $\epsilon \rightarrow
0$.  As $\epsilon$ and $N$ are arbitrary, if we can prove that the
divergence assumption implies that $W_0(m,n; \psi)$ is full inside
$A_{\epsilon,N}$, this will give the full result.

For an $m\times n$-matrix $X$, let $\tilde{X}$ denote the $n \times
n$-matrix formed by the first $n$ rows. We will be considering a set
for which $\tilde{X}$ is invertible. Evidently, the exceptional set is
of measure zero within $\mathbb{R}^{mn}$. However, to make things
work, we will need to work with the set
\begin{equation*}
  A_{\epsilon, N} = \left\{X \in \mat_{m\times n}(\mathbb{R}) : \epsilon <
    \det(\tilde{X}) < \epsilon^{-1}, \quad \max_{1 \leq i,j\leq n}
    \abs{x_{ij}} \leq N \right\}
\end{equation*}
The set is of positive measure for $\epsilon$ small enough and $N$
large enough, and as $\epsilon$ decreases and $N$ increases, the set
fills up $\mat_{m\times n}(\mathbb{R})$ with the exception of the
null-set of matrices $X$ such that $\tilde{X}$ is singular.

We will translate the statement about small linear forms to one
about usual Diophantine approximation. This will allow us to
conclude from a Khintchine--Groshev theorem. We may rewrite the
$X$ as
\begin{equation*}
  X =
  \begin{pmatrix}
    \tilde{X} \\
    X'
  \end{pmatrix} =
  \begin{pmatrix}
    I_n \\
    \hat{X}
  \end{pmatrix} \tilde{X},
\end{equation*}
where $X'$ denotes the matrix consisting of the last $m-n$ rows of the
original matrix and $\hat{X}$ denotes the matrix $X' \tilde{X}^{-1}$.

Consider the set of $n$ linear forms in $m-n$ variables defined by the
matrix $\hat{X}$. Suppose furthermore that these linear forms satisfy
the inequalities
\begin{equation}
  \label{eq:1}
  \norm{\mathbf{r} \hat{X}}_i \leq
  \frac{\psi(\abs{\mathbf{r}})}{nN}, \quad 1 \leq i \leq n,
\end{equation}
for infinitely many $\mbf{r} \in \mathbb{Z}^{m-n} \setminus \{\mbf
0\}$, where $\norm{\mathbf{x}}_i$ denotes the distance from the $i$'th
coordinate of $\mathbf{x}$ to the nearest integer. A special case of
Khintchine--Groshev states, that the divergence condition of our
theorem implies that the set of such linear forms $\hat{X}$ is full
inside the set of $(m-n) \times n$-matrices, and hence in particular
also in the image of $A_{\epsilon, N}$ under the map sending $X$ to
$\hat{X}$.

Now, suppose that $X \in A_{\epsilon, N}$ is such that $\hat{X}$ is
in the set defined by \eqref{eq:1}. We claim that $X$ is in
$W_0(m,n:\psi)$. Indeed, let $\mathbf{r}_k$ be an infinite sequence
such that the inequalities \eqref{eq:1} are satisfied for each $k$,
and let $\mathbf{p}_k$ be the nearest integer vector to
$\mathbf{r}_k \hat{X}$. Now define $\mathbf{q}_k = (\mathbf{p}_k,
\mathbf{r}_k)$. The inequalities defining $W_0(m,n,\psi)$ will be
satisfied for these values of $\mathbf{q}_k$, since
\begin{equation*}
  \abs{\mathbf{q}_k X} = \abs{\mathbf{q}_k
  \begin{pmatrix}
    I_n \\
    \hat{X}
  \end{pmatrix} \tilde{X}} = \abs{(\pm \norm{r_k \hat{X}}_1, \dots,
  \pm \norm{r_k \hat{X}}_n) \tilde{X}}.
\end{equation*}
The $i$'th coordinate of the first vector is at most
$\psi(\abs{q})/nN$, so carrying out the matrix multiplication, using
the triangle inequality and the fact that $\abs{x_{ij}} \leq N$ for
$1 \leq i, j \leq n$ shows that
\begin{equation*}
  \abs{\mathbf{q}_k X}_i < \psi(\abs{\mathbf{q}_k}).
\end{equation*}
Applying Lemma \ref{lem:bi-lip}, the divergence part of Theorem
\ref{thmmgn} follows in the case of Lebesgue measure.

\subsection{Infinite measure}

The infinite measure case of the Theorem \ref{thmmgn} can be easily
deduced from the following lemma.

\begin{lem}
  \label{lem:infmeslem}
  Let $\psi$ be an approximating function and let $f$ and $g:r\to
  r^{-n^2}f(r)$ be dimension functions with $r^{-mn}f(r)\to\infty$ as
  $r\to{}0$. Further, let $r^{-(m-n-1)n}g(r)$ be a dimension function
  and $r^{-(m-n)n}g(r)$ be monotonic. If
  \[
  \sum \limits_{r=1}^{\infty}f(\Psi
  (r))\Psi(r)^{-(m-1)n}r^{m-1}=\infty,
  \]
  then
  \[
  \mcal{H}^{f}(W_0(m, n; \psi))=\infty.
  \]
\end{lem}

\begin{proof}
  We define a Lipschitz map to transform our problem to a classical
  one. As in the finite measure case, we fix $\epsilon > 0$, $N \geq
  1$ and let
  \begin{equation*}
    A_{\epsilon, N} = \left\{X \in \mat_{m\times n}(\mathbb{R}) :
      \epsilon < \det(\tilde{X}) < \epsilon^{-1}, \quad \max_{1 \leq
        i,j\leq n} \abs{x_{ij}} \leq N \right\},
  \end{equation*}
  where $\tilde{X}$ denotes the $n \times n$-matrix formed by the first
  $n$ rows. We also define the set
  \begin{equation*}
    \tilde{A}_{\epsilon, N} = \left\{\tilde{X} \in \GL_{n}(\mathbb{R}) :
      \epsilon < \det(\tilde{X}) < \epsilon^{-1}, \quad \max_{1 \leq
        i,j\leq n} \abs{x_{ij}} \leq N \right\}.
  \end{equation*}
  For an appropriately chosen constant $c>0$ depending only on
  $m,n,\epsilon$ and $N$, we find that the map
  \begin{equation}
    \label{eq:2}
    \eta : W(m-n, n, c \psi) \times \tilde{A}_{\epsilon, N} \rightarrow
    W_0(m,n,\psi), \quad (Y, X) \mapsto
    \begin{pmatrix}
      X \\
      YX
    \end{pmatrix},
  \end{equation}
  is a Lipschitz embedding. Indeed, it is evidently injective as $X$
  is invertible for all elements of the domain. The Lipschitz
  condition follows as we have restricted the determinant to being
  positive. That the image is in $W_0(m,n,\psi)$ follows by
  considering the system of inequalities as above and choosing $c > 0$
  accordingly as above in the finite measure case. Consequently, the
  map $\eta$ is bi-Lipschitz onto its image. We have,
  \begin{eqnarray*} \mcal{H}^{f}(W_0(m, n; \psi))
    &\geq&\mcal{H}^f\left(\eta\left(W(m-n,n;\psi)\times \tilde{A}_{
          \epsilon, N}\right)\right) \notag \\
    &\asymp&\mcal{H}^{f}\left(W(m-n,n;\psi)\times \tilde{A}_{
        \epsilon, N}\right).
  \end{eqnarray*}
  
  The main idea of the proof is now to apply slicing, refer to
  Lemma~\ref{slicing}, to a tranlate of the Borel set
  $B:=W(m-n,n;\psi)\times \tilde{A}_{\epsilon, N} \subseteq
  \mb{I}^{mn}$. Initially, we fix an arbitrary point $X_0\in
  \tilde{A}_{e, N}$. Let $\sigma: \mathbb{R}^{n^2} \rightarrow
  \mathbb{R}^{n^2}$ be the translation map sending $X_0$ to the
  origin, \emph{i.e.}, $\sigma(X) = X-X_0$. Let $\tilde{\sigma}:
  \mathbb{R}^{mn} \rightarrow \mathbb{R}^{mn}$ be the map which leaves
  the upper $(m-n) \times n$ matrix untouched but applies $\sigma$ to
  the lower $n \times n$ matrix. We apply these maps to all sets
  above. This leaves Hausdorff measure invariant, so by abuse of
  notation we will denote the translated sets by the same letters as
  the original ones.

  Let $V$ be the space
  \[\mathbb{I}^{(m-n)n}\times{}\{0\}^{n^2}.\]
  Let \[S=V^\perp:=\{0\}^{(m-n)n}\times\mb{I}^{n^2}\] and further it
  has positive $\mcal{H}^{n^2}$-measure. Now for each $b\in{}S$
  \begin{eqnarray*}
  \mcal{H}^{g}\left(B\cap(V+b)\right)&=&\mcal{H}^{g}
  \left((W(m-n,n;\psi)\times \tilde{A}_{\epsilon,
      N})\cap(V+b)\right)\notag\\
  &=&\mcal{H}^{g}\left(\left(W(m-n,n;\psi)\times
      \{0\}^{n^2}\right)+b\right)\notag\\&\asymp &
  \mcal{H}^{g}(W(m-n,n;\psi)) \label{3.27} \\ &=& \infty \ \ \text{if}
  \ \ \hfill \sum\limits_{r=1}^\infty
  g(\Psi(r))\Psi(r)^{-(m-n-1)n}r^{m-1}=\infty.
\end{eqnarray*}

\noindent The slicing lemma yields that
\[\mcal{H}^{f}\left(W(m-n,n;\psi)\times
  \tilde{A}_{\epsilon, N}\right)=\infty \ \ \text{if} \
\sum\limits_{r=1}^\infty
g(\Psi(r))\Psi(r)^{-(m-n-1)n}r^{m-1}=\infty.\]

\noindent Since, $g:r\to r^{-n^2}f(r)$, we have

\begin{equation*}\label{3.30}
\mcal{H}^{f}(W_0(m,n;\psi))=\infty \ \ \textrm{if}\ \
\sum\limits_{r=1}^\infty f(\Psi(r))\Psi(r)^{-(m-1)n}r^{m-1}=\infty.
\end{equation*}

\end{proof}

\section{Proof of Theorem \ref{thmmlen}}

The method of proof of Theorem \ref{thmmlen} is similar to Theorem $2$
of \cite{mhjl} which relies mainly on Theorem \ref{thmmgn} and the
slicing technique. To be brief, one first shows that for $m \leq n$,
the set $W_0(m,n;\psi)$ must be contained in a hypersurface of dimension at
most $(m-1)(n+1) < mn$. This follows by proving that if $X \in
W_0(m,n;\psi)$, then the columns of $X$ must be linearly
dependent.

This required linear dependence can be removed from the problem by
introducing another bi-Lipschitz map to a non-degenerate setting. The
above methods applied for the case $m > n$ can then be applied to the
non-degenerate setting in both the case of finite and infinite
measure. The details are essentially the same as those in \cite{mhjl},
and are left to the interested reader.

\section{Concluding remarks}

In this paper we have made no effort to remove monotonic assumption on
the approximating function to prove the analogue of
Khintchine--Groshev theorem for the absolute value setup. However
there are still some open territories not investigated in this paper.
To conclude the paper we discuss them here.

Theorem \ref{thmmgn} provides a beautiful `zero--full' criterion under
certain divergent sum conditions but on the other hand Theorem
\ref{thmmlen} provides `zero--positive' criterion. The later theorem
relies on taking the linear combinations of the independent vectors
from the former but the combinations does not span the full space. It
is natural to conjecture that a `zero--full' law for Theorem
\ref{thmmlen} does indeed hold.

In the current paper settings the approximating function $\psi$ is
dependent on the supremum norm of the integer vector $\mbf q$.
Clearly, a natural generalization is to consider multivariable
approximating function, $\Psi:\mb Z^m \to \mb R^+$ and their
associated set $W_0(m, n; \Psi)$.

Another natural generalisation is the case of different rates of
approximation for each coordinate, \emph{i.e.}, when we consider
inequalities
\begin{equation*}
  |q_{_{1}}x_{_{1}i}+q_{_{2}}x_{_{2}i}+\dots+q_{m}x_{mi}| \leq
  \psi_i(|\mbf{q}|)\text{ for }(1\leq i\leq n),
\end{equation*}
where the $\psi_i$ are potentially different error functions. In this
case, it is shown in \cite{fishler} that the analogue Corollary
\ref{cor1} holds with $\psi(r)^n$ in the series replaced by $\psi_1(r)
\cdots \psi_n(r)$.

\def\cprime{$'$}
\providecommand{\bysame}{\leavevmode\hbox to3em{\hrulefill}\thinspace}
\providecommand{\MR}{\relax\ifhmode\unskip\space\fi MR }
\providecommand{\MRhref}[2]{%
  \href{http://www.ams.org/mathscinet-getitem?mr=#1}{#2}
}
\providecommand{\href}[2]{#2}

\end{document}